\documentclass[a4paper,11pt]{article}
\usepackage{mathtools}
\mathtoolsset{showonlyrefs}
\usepackage{amsmath}
\usepackage{amssymb}
\usepackage{amsthm}
\usepackage[utf8]{inputenc}
\usepackage[ruled]{algorithm2e}
\usepackage{etex}

\SetKw{KwInit}{Initialization:}
\SetKw{KwBreak}{break}
\SetKw{KwInput}{Input:}
\SetKw{KwOutput}{Output:}
\SetKw{KwIter}{Iteration:}

\usepackage{graphicx,subcaption}
\usepackage{pgfplots}
\usepackage{url}

\makeatletter
\g@addto@macro\@floatboxreset\centering
\makeatother

\newcommand{\N}{\ensuremath{\mathbb{N}}}

\newcommand{\R}{\ensuremath{\mathbb{R}}}

\newcommand{\tT}{{\scriptscriptstyle \operatorname{T}}}
\DeclareMathOperator*{\argmin}{arg\,min}
\DeclareMathOperator*{\argmax}{arg\,max}

\usepackage{tikz,xargs}
\usetikzlibrary{spy}
\usetikzlibrary{positioning}
\usetikzlibrary{datavisualization}

\title{
Segmentation of FIB data with Brightness Variations
}

\newtheorem{remark}{Remark}
\newtheorem{proposition}[remark]{Proposition}
\newtheorem{theorem}[remark]{Theorem}
\newtheorem{corollary}[remark]{Corollary}

\title{
	Unsupervised Multi Class Segmentation 
	of 3D Images with Intensity Inhomogeneities
}

\author{Jan Henrik Fitschen\footnote{Department of Mathematics,
		University of Kaiserslautern, Germany,
		{\{fitschen,steidl\}@mathematik.uni-kl.de}}
	\and 
	Katharina Losch\footnote{Fraunhofer ITWM, Kaiserslautern, Germany}
	\and Gabriele Steidl\footnotemark[1]
}

\begin{document}
	\maketitle
\begin{abstract}
\noindent
Intensity inhomogeneities in images cause problems in gray-value based image segmentation since the varying intensity often dominates over gray-value differences of the image structures.
In this paper we propose a novel biconvex variational model
that includes the intensity inhomogeneities
to tackle this task.
We combine a total variation approach for 
multi class segmentation 
with a multiplicative model to handle the inhomogeneities.
In our model we assume that the image intensity 
is the product of a smoothly varying part and a component 
which resembles important image structures such as edges.
Therefore, we penalize in addition to the total variation of the label assignment matrix 
a quadratic difference term to cope with the smoothly varying factor.
A critical point of the resulting biconvex functional is computed by a modified
proximal alternating linearized minimization method
(PALM). We show that the assumptions for the convergence of the
algorithm are fulfilled.
Various numerical examples demonstrate the very good performance 
of our method. Particular attention is paid to the segmentation of 
3D FIB tomographical images serving as a motivation for our work.
\end{abstract}

%
\section{Introduction}
%
Intensity inhomogeneities often occur in real-world images mainly due to different spatial
lighting and deficiencies of imaging devices.
For example in MRI, imperfections in the radio-frequency coils 
or problems associated with acquisition sequences cause intensity changes.
The motivation for this paper was the task of segmenting 3D images 
stemming from focused ion beam (FIB) tomography.
While classical $X$-ray tomography does often not reach the required material resolution,
FIB tomography enables to investigate the 3D morphology of structures on a scale down to several nanometers.
The material is successively removed by a focused ion beam and after every section, the surface is displayed by scanning electron microscopy. 
Several hundred of these serial slices finally form a 3D image. 
A typical slice of a 3D FIB tomography of  aluminum with  silicon carbide (SiC) particles (larger black parts)
and copper aggregations (small white parts) is shown in Fig.~\ref{fig:bad_ex:a}.
The segmentation has to distinguish between the particles, the aggregates and the surrounding aluminum matrix.
However, due to the intensity inhomogeneities, a segmentation based on the gray-values gives a very bad result.
Fig.~\ref{fig:bad_ex:b} and \ref{fig:bad_ex:c} show segmentation results for a 3D FIB data set using a supervised gray-value based segmentation method without considering the illumination.
Therefore, we have to choose one cluster center for each of the three classes, i.e., in total three gray-values that are close to the classes aluminum, SiC and copper, respectively.
Unfortunately, the cluster centers cannot be chosen appropriate for all parts of the image so that either too many or
not enough particles are detected.
Therefore the segmentation of such images has to take the intensity inhomogeneities into account.
\begin{figure*}
	\begin{subfigure}[t]{0.42\textwidth}\centering
		\includegraphics[width=0.7\textwidth]{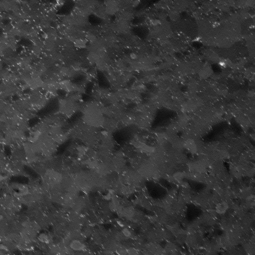}
		\caption{Exemplary slice of a 3D data set with intensity inhomogeneities}
		\label{fig:bad_ex:a}
	\end{subfigure}\\
	\begin{subfigure}[t]{0.48\textwidth}\centering
		\includegraphics[width=0.99\textwidth]{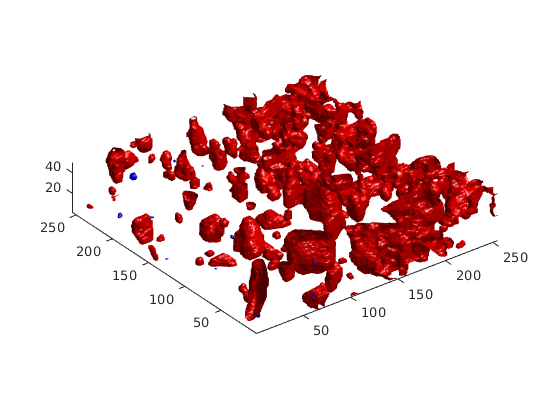}
		\caption{Segmentation without considering intensity inhomogeneities with cluster centers chosen such that the left part is segmented correctly.}	
		\label{fig:bad_ex:b}
	\end{subfigure}	\hspace{0.2cm}
	\begin{subfigure}[t]{0.48\textwidth}\centering
		\includegraphics[width=0.99\textwidth]{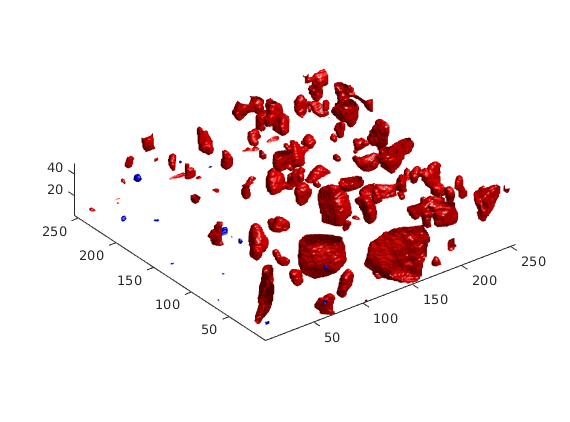}
		\caption{Segmentation without considering intensity inhomogeneities with cluster centers chosen such that the right part is segmented correctly.}
		\label{fig:bad_ex:c}
	\end{subfigure}
	\caption{One slice of FIB data with varying illumination 
	and two 3D segmentation results using a model not considering the illumination.}\label{fig:bad_ex}
\end{figure*}

There are several techniques for illumination corrections 
in the literature. These methods could be used
in a preprocessing step before applying a segmentation algorithm of choice.
In particular in MRI, intensity corrections were proposed by
simple homomorphic filtering \cite{BMR98,JAMA96} and 
polynomial, resp., spline surface fitting approaches
\cite{DZM93,MBP95, TMGW93,GBDTA96}.
Many spatial illumination correction methods for natural images take hypotheses 
about the Human Visual System (HVS) into account.  
In particular, the perceptual work about the Retinex model \cite{LM71} has found wide acceptance.
It states that the HVS does not perceive an absolute lightness but rather a relative, local one.
This phenomenon is called lateral inhibition.
For example, Fig.~\ref{fig:chess} shows the experiment 
of the checkerboard shadow illusion of Adelson \cite{adelson1995checkershadow}.
Although the squares A and B have the same gray-value, the perceived intensities are different.
In the Retinex model, the light intensity $F$ perceived by the observer or camera is considered to be 
the product of the reflectance of the objects $R$ in the scene 
and the amount of source illumination  $L$ falling on the objects, see also \cite[p.~51]{GW08},
\begin{equation} \label{retinex}
F(x) = R(x) L(x),
\end{equation}
where $R(x) \in (0,1)$ and $L(x) \in (0,+\infty)$.
While we assume that the reflectance inherits the structures of the objects, e.g.~edges,
we consider the illumination as spatially smooth, in particular it should not have sharp edges, which represent the image structures. 
Taking the logarithm in \eqref{retinex} we obtain
\begin{equation} \label{retinex_log}
f(x) = r(x) + l(x),
\end{equation}
where $f = \log F$, $r = \log R$ and $l = \log L$.
Retinex based variational or PDE based approaches for illumination corrections 
can be found for example in \cite{LZ15,MMOC11,MPS10,NW11}.

In MRI the observed intensity is often modeled similarly as in \eqref{retinex}
by the product
of a structural part $R$ and a so called gain factor $L$.
In this paper we also follow the  multiplicative intensity model.

In contrast to a two step procedure we consider
the simultaneous segmentation and intensity inhomogeneity estimation.
This avoids the computational burden of two separate procedures
and has moreover the advantage of being able to use intermediate information
from the segmentation while performing the update.
Besides statistical methods as the computationally extensive EM approach in \cite{WGKJ96},
variational based algorithms were proposed in \cite{AYMFM02,ABCAK16,LHDGMG08,LHDGMG11,MS07,PP99,ZZLZ16}.
We consider the later approaches in more detail in the next section.

Variational segmentation models have shown a very good performance 
and flexibility in many applications.
Level set methods and convex segmentation models which penalize the (nonlocal) discrete total variation (TV)
of a relaxed label assignment matrix have been successfully applied 
\cite{BYT09,BC08,CV01,LBS09,LKYBS09,PCCB09,SS12a,ZGFN08}.

In this paper, we combine the TV based segmentation method with the 
multiplicative intensity model.
This results in a biconvex non smooth functional, which has to be optimized
with respect to the label assignment matrix, the cluster centers 
and the smoothly varying intensity factor.
Then we obtain both a segmented image and an estimation of the intensity inhomogeneities.
We compute a critical point of the corresponding functional by applying a slight modification
of the proximal alternating linearized minimization method
(PALM) by Bolte et al.~\cite{BST14}. 
\begin{figure*}
	\centering
\begin{subfigure}[t]{0.32\textwidth}\centering
	\includegraphics[width=0.9\linewidth]{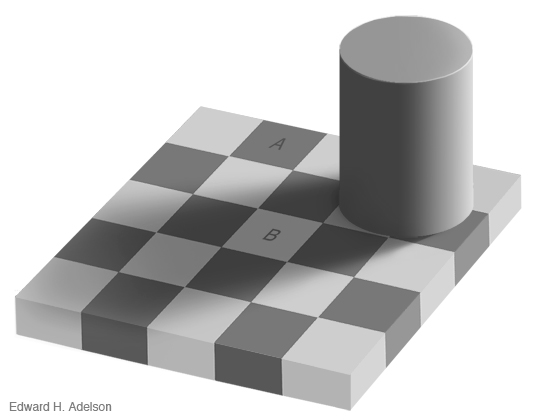}
	\caption{Original image}
\end{subfigure}
\begin{subfigure}[t]{0.32\textwidth}\centering
	\includegraphics[width=0.9\linewidth]{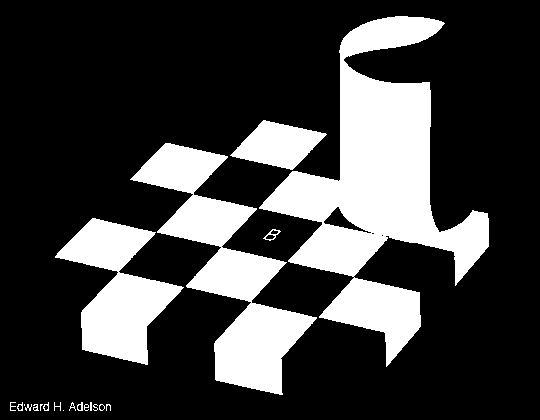}
	\caption{Segmentation result}
\end{subfigure}
\begin{subfigure}[t]{0.32\textwidth}\centering
	\includegraphics[width=0.9\linewidth]{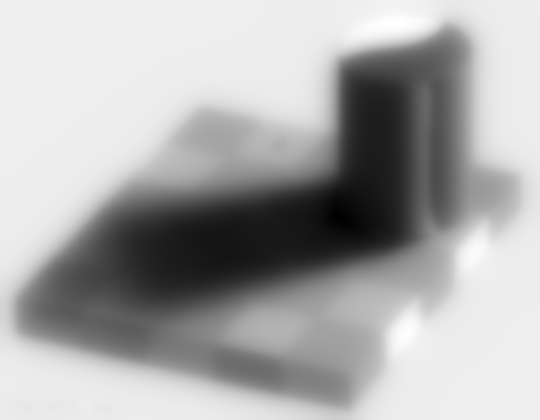}
	\caption{Computed illumination}
\end{subfigure}
	\caption{Result for the ``checkerboard'' image.}
	\label{fig:chess}
\end{figure*}	

The paper is organized as follows:
In Section \ref{sec:model} we review variational segmentation models
which compensate for the varying intensities while segmenting the image.
Then we introduce our model. 
The modified PALM is explained in Section \ref{sec:palm}.
Section \ref{sec:numerics} contains numerical examples.
The paper finishes with conclusions in Section \ref{sec:conclusions}.
%
\section{New Model} \label{sec:model}
Let ${\cal G} \coloneqq \{1,\ldots, n_1\} \times \ldots \times \{1,\ldots, n_d\}$
be a $d$-dimensional image grid.
Here we will deal with $d=2$ and $d=3$.
Let $n= n_1 \ldots n_d$ be the number of image pixels.
We consider images $F\colon {\cal G} \rightarrow \R$ which we want to segment into $K$ classes.
We introduce a so-called label assignment matrix  
$$
u \coloneqq (u_k(j))_{ j \in {\cal G}, k \in \{ 1,\ldots,K \} }.
$$
Ideally we would like to have for a fixed pixel $j \in {\cal G}$ that $u_k(j) = 1$ if it belongs to class $k$
and $u_k(j) = 0$ otherwise.
However, in the subsequent variational approach this would lead to optimization tasks which are
NP hard to solve. 
Therefore it is common to relax the assumptions on the label assignment matrix
and  require only that $u_k(j) \in [0,1]$
is the probability that pixel $j$ belongs to class $k$.
Since every pixel should be assigned to one of the classes, we enforce for every $j \in {\cal G}$ that
$$
\sum_{k=1}^K u_k(j) = 1.
$$
In other words $(u_k(j))_{k=1}^K$ is an element of the probability simplex
\begin{equation}\label{prob_simplex}
 \triangle_K \coloneqq \{ v =(v_k)_{k=1}^K \in [0,1]^K: \sum_{k=1}^K v_k = 1\}.
\end{equation}
From the label assignment matrix the segmentation can be obtained by
assigning, e.g., the label $$\hat k \coloneqq \argmax_{k=1,\ldots,K} u_k(j)$$ to $j \in {\cal G}$.

A variational model is composed of a data term, which incorporates information about
the given image $F$ and penalizing/regularizing terms, which contain prior information on the image.
Let $C \coloneqq (C_1,\ldots,C_K)^\tT$ denote the unknown vector of the centers of the $K$ gray-value clusters.
For segmentation problems a typical data term is given by
\begin{equation} \label{data_start}
 \sum_{k=1}^K \sum_{j \in {\cal G}} (u_k(j))^p |F(j) - C_k|^2,
\end{equation}
where $p \ge 1$.
For $p=1$ this data term appears in the $K$ (or $C$) means approach,
while $p>1$ is related to the fuzzy $K$-means method \cite{Bezdek81a,Bezdek81,HHMSS12,MA96}. 
Roughly speaking, the data term indicates the following:
If $F(j)$  is close to the center $C_k$, then the corresponding summand in \eqref{data_start} 
remains small for a larger $u_k(j)$.
Thus the probability of $F(j)$ being in the class $k$ is high.
If $F(j)$  is  far away from the center $C_k$, then the corresponding term in \eqref{data_start} 
becomes only small if the probability 
$u_k(j)$ that $F(j)$ belongs to class $k$ is small.
Clearly other distance measures between $F(j)$ and $C_k$ than the squared absolute value can be chosen.
Moreover, the approach can be generalized to color images or more general image feature vectors $F$
than gray-values. In this paper we restrict our attention to gray-value images.
The above data term does not take intensity inhomogeneities into account.
However, as shown in Fig.~\ref{fig:bad_ex} it is often not possible to find appropriate
class centers if the illumination varies.
A remedy consists in incorporating the smoothly varying intensity part $L$
into the model and consider
\begin{equation}\label{seg_mult_1}
{\cal E}_{data} (u,C,L) \coloneqq  \sum_{k=1}^K \sum_{j \in {\cal G}} (u_k(j))^p |F(j) - L(j) C_k|^2,
\end{equation}
$p \ge 1$,
or its logarithmic counterpart
\begin{equation}\label{seg_mult_2}
E_{data} (u,c,l) \coloneqq  \sum_{k=1}^K \sum_{j \in {\cal G}} (u_k(j))^p |f(j) - l(j) - c_k|^2,
\end{equation}
where $f = \log F$, $c = \log C$ and $l = \log L$.
From \eqref{seg_mult_2} we get the illumination by $L = \exp(l)$ and the class centers by $C = \exp(c)$.
The functional \eqref{seg_mult_1} is triconvex, i.e., fixing two of the variables $u,C,L$
the functional is convex in the third variable.
In contrast, the functional \eqref{seg_mult_2} is only biconvex in $u$ and $v \coloneqq (c,l)$.
This means if we fix $u$, then the functional is convex in $v$ and conversely.

Within the penalizers pixel differences are often used since they indicate smooth or higher frequent image parts
as edges.
By $\nabla F$ we denote the $d$-dimensional discrete gradient of $F$, 
where we use forward differences in each
direction together with mirror boundary conditions.
For the concrete matrix structure of the gradient operator we refer to \cite{SS12a}.
Similarly, $\nabla^2 F$ is the discrete Hessian of $F$.
Throughout this paper $\|A\|_2$ denotes the square root of the sum of the 
squared entries of a multidimensional array $A$.

Let us briefly review existing variational models for image segmentation
which take intensity inhomogeneities into account.
In \cite{PP99} the following model with data term \eqref{seg_mult_1} with $p=2$ was suggested:
\begin{align}\label{new_model_cb}
	&
	\argmin_{u,C,L} \big\{ {\cal E}_{data} (u,C,L) + 
	+ \gamma_1 \Big\| \nabla L \Big\|^{2}_{2}
	+ \gamma_2 \Big\| \nabla^2 L 
	\Big\|^2_2 \big\}
	\\ 
	&\qquad \qquad \mathrm{subject \; to}  \quad \left(u_k(j) \right)_{k=1}^K \in \triangle_K, 
		\;	j \in {\cal G},
\end{align}
where $\gamma_i >0$, $i=1,2$ are regularization parameters.

This model only penalizes non smooth intensity inhomogeneities $L$ via two terms,
but does not take the edges of the label assignment matrix $u$ into account.
Therefore it is not robust to noise. 
Moreover it has the drawback that there does not exist a minimizer.
This can be simply seen by setting 
$L \coloneqq F/r$ and $C_k \coloneqq r$ with some constant $r>0$ so that the functional becomes 
$
\frac{\gamma_1}{r} \Big\| \nabla L \Big\|^{2}_{2}
	+ \frac{\gamma_2}{r^2} \Big\| \nabla^2 L \|_2^{2}
$.
For $r \rightarrow +\infty$ we obtain a minimizing sequence of the functional which does not converge.

Another segmentation model for MRI with data term \eqref{seg_mult_2}
was proposed by Ahmed et al.~\cite{AYMFM02}:
\begin{align}\label{model_ahmed}
&\argmin_{u,C,L} \Big\{  E_{data}(u,c,l) \\
&\qquad \qquad \qquad + 
\lambda \sum_{k=1}^K \sum_{j \in {\cal G}} (u_k(j))^p  \sum_{i \in {\cal N}_j} (f(i) - l(i) - c_k)^2\Big\}
\\
& \qquad \qquad \mathrm{subject \; to}  \quad \left(u_k(j) \right)_{k=1}^K \in \triangle_K, \;	j \in {\cal G},\\
& \qquad \qquad \qquad \qquad \quad \;  0< \sum_{j \in {\cal G}} u_k(j) < n , \; k=1,\ldots,K.
\end{align}
Here $\lambda >0$ is a regularization parameter. 
The second term takes the illumination in the neighborhood ${\cal N}_j$ of the $j$-th pixel 
into account in order to respect the lateral inhibition. 
However, as in the previous model no penalizer for the label assignment matrix is used,
which makes the model sensitive to noise.
A similar model which penalizes the simplex constraint 
was suggested in \cite{MS07}. 
Unfortunately, the model formulation in \cite{MS07} is mathematically not sound 
and the optimization procedure does not fit to the model.

Li et al.~\cite{LHDGMG08,LHDGMG11} considered a variational level set model,
which reads in the continuous setting as
\begin{align} \label{model_li}
\argmin_{\phi,C,L} \Big\{  \int \sum_{k=1}^K &e_k(x) M_k(\phi(x)) \, dx\\
&+ \lambda \int |\nabla H(\phi)| \, dx 
+ \gamma \int (|\nabla \phi |-1)^2 \, dx \Big\}
\end{align}
with 
\begin{align}
	e_k(x) \coloneqq \int K_\sigma(y-x) |F(x)- L(y) C_k)|^2 \, dy.
\end{align}
Here $K_\sigma$ is a truncated Gaussian kernel with standard deviation $\sigma$, $H$ a smoothed Heaviside function and $M_k$  is a membership function.
The model was applied for two-dimensional medical images.
Note that the functional is not convex in $\phi$ for fixed $C$ and $L$.
A more general level set approach is considered in \cite{ZZLZ16}.
Furthermore, a slightly different approach was proposed also in \cite{ABCAK16}.

In this paper we consider the model
\begin{align}\label{new_model_log}
	&\argmin_{u,c,l} \Big\{E_{data}(u,c,l) 
	+  \sum_{k=1}^K \lambda_k \Big\| | \nabla u_k |	\Big\|_1 
	+ \gamma  \Big\| \nabla l		\Big\|_2^2 \Big\}
	\\ 
	&\qquad	\qquad \mathrm{subject \; to}  \quad \left(u_k(j) \right)_{k=1}^K \in \triangle_K, 
		\;	j \in {\cal G},
\end{align}
where $p=1$ in the data term. 
Here $\lambda >0$ and $\gamma >0$ are regularization parameters and 
$$u_k \coloneqq \left( u_k(j) \right)_{j \in {\cal G}}$$
is the $k$-th label assignment matrix.
We have not found this model in the literature.
Using the indicator function of sets
$$
\iota_S(x) \coloneqq 
\left\{
\begin{array}{ll}
                       0&{\rm if } \; x \in S,\\
                       +\infty&{\rm otherwise} ,
                                \end{array}
                                \right.
$$
we can express the constraint ``subject to $\left(u_k(j) \right)_{k=1}^K \in \triangle_K,\;	j \in {\cal G}$'' by adding $\iota_{\triangle_K^n}(u)$ to the functional.
Hence, the functional in \eqref{new_model_log} can be written as
\begin{align}\label{new_model_func}
E(u,c,l) =   E_{data}(u,c,l) 
	&+  \sum_{k=1}^K \lambda_k \Big\| | \nabla u_k |	\Big\|_1 \\
	&+ \gamma  \Big\| \nabla l		\Big\|_2^2
	+ \iota_{\triangle_K^n}(u).
\end{align}
As in \eqref{new_model_cb} we penalize a non smooth illumination $l$, but just by one quadratic first order difference term.
The matrix $u_k$ is penalized by its discrete total variation (TV) \cite{ROF92},
where $|\nabla u_k|$ denotes the length of the discrete gradient of $u_k$ and the $\ell_1$ norm
is taken over the grid points $j \in {\cal G}$. This term encourages smooth edges
of the segmented objects and makes the model more robust to noise.
We allow the choice of different parameters $\lambda_k$ to better segment objects of different size.
The functional \eqref{new_model_log} is biconvex in $u$ and $(c,l)$ and possesses a minimizer as the following proposition states.
\begin{proposition}
	The functional $E(u,c,l)$ defined by
	\eqref{new_model_func} has a minimizer.
\end{proposition}
\begin{proof}
	It suffices to show that there exists a bounded infimal sequence
	since $E(u,c,l)$ is  lower semi-continuous and bounded from below.
	
	Let $((u,c,l)^r)_{r \in \mathbb{N}}$, $(u,c,l)^r \in \triangle_K^n \times \R^K \times \R^n$, be an infimal sequence, i.e.,
	\begin{align}
	E(u^r,c^r,l^r) \rightarrow \inf_{(u,c,l)} E(u,c,l) \quad \mathrm{as} \ r \rightarrow \infty.
	\end{align}    
	
	Without loss of generality, we can assume the following: If there exists a subsequence $(u^{r_i}_k)_{i\in \N}$ for some $k \in 1,\ldots,K$ with 
	\begin{align}\label{eq:segm_0}
	\sum_{j\in\mathcal{G}} u^{r_i}_k(j) \rightarrow 0 \ \mathrm{as} \ i \rightarrow \infty,
	\end{align}
	then the whole sequence converges to zero, in other words
	\begin{align}\label{eq:segm_0b}
	\sum_{j\in\mathcal{G}} u^{r}_k(j) \rightarrow 0 \ \mathrm{as} \ i \rightarrow \infty.
	\end{align}
	If this was not the case, we could restrict the following analysis to the sequence $((u,c,l)^{r_i})_{i \in \mathbb{N}}$ and possibly repeat the procedure for the next $k$ fulfilling \eqref{eq:segm_0}.
	Note that it is not possible that \eqref{eq:segm_0b} holds for all $k \in \{1,\ldots,K\}$ since $u^r \in \triangle_K^n$.
	
	Since the sequence is assumed to be an infimal sequence, 
	we immediately obtain by the constraint on $u^r$ that $(u^r)_{r \in \mathbb{N}}$ is bounded.
	
	Next we consider $l$.
	Due to mirror extension at the boundary, we have that 
	\begin{align}
	\ker \nabla = \{a \mathbf{1}_n : a \in \R \},
	\end{align}
	where $\mathbf{1}_n$ is the vector consisting of $n$ entries $1$,
	i.e., the kernel consists of constant images.
	We decompose $l$ into two parts
	\begin{align}
	l^r = l_{\ker}^r \mathbf{1}_n + l_{\ker^\perp}^r,
	\end{align}
	where $l_{\ker^\perp}^r \in \ker(\nabla)^\perp$ and $l_{\ker}^r \mathbf{1}_n \in \ker(\nabla)$ with $l_{\ker}^r \in \mathbb{R}$.
	Using again that the sequence is an infimal  one, we see due to the term $\| \nabla l\|_2^2$ 
	that $(l^r_{\ker^\perp})_{r \in \mathbb{N}}$ is bounded.
	
	Now define the (possibly empty) set 
	\begin{align}
	\mathcal{K} \coloneqq \left\{k : \sum_{j\in\mathcal{G}} u^r_k(j) \rightarrow 0 \ \mathrm{as} \ r \rightarrow \infty\right\} \subset \{1,\ldots,K\}
	\end{align}
	and fix $\hat k \not \in \mathcal{K}$.
	We show that a bounded infimal sequence is given by $((\tilde u,\tilde c,\tilde l)^r)_{r \in \mathbb{N}}$, where
	\begin{align}
	\tilde u_k^r &\coloneqq \begin{cases}
	\mathbf{0} 	& \mathrm{if } \ k \in\mathcal{K}, \\
	u_k^r + \sum_{i \in\mathcal{K}} u_i^r & \mathrm{if } \ k=\hat k,\\
	u_k^r		& \mathrm{otherwise},
	\end{cases}\\
	\tilde c_k^r &\coloneqq \begin{cases}
	0 	& \mathrm{if } \ k \in\mathcal{K}, \\
	c_k^r+l^r_{\ker} & \mathrm{otherwise},
	\end{cases}\\	
	\tilde l^r &\coloneqq l^r-l^r_{\ker}\mathbf{1}_n = l_{\ker^\perp}^r.	
	\end{align}
	We immediately obtain that $(\tilde u^r)_{r \in \mathbb{N}}$ is bounded as $(u^r)_{r \in \mathbb{N}}$ is bounded.
	Further, $(\tilde l^r)_{r \in \mathbb{N}} = (l^r_{\ker^\perp})_{r \in \mathbb{N}}$ is bounded.
	Next, we show that $(\tilde c^r)_{r \in \mathbb{N}}$ is bounded.
	From the data term \eqref{seg_mult_2}, we conclude that 
	\begin{align}
	(\sum_{j\in\mathcal{G}} u_k^r(j) |f(j) - l^r(j) - c^r_k|^2)_{r \in \mathbb{N}}
	\end{align} 
	is bounded for all $k=1,\ldots,K$.
	Let $k \not \in \mathcal{K}$. By our assumption that \eqref{eq:segm_0} implies \eqref{eq:segm_0b}, we know that there also does not exist a subsequence $(r_i)_{i \in \N}$ with $\sum_{j\in\mathcal{G}} u^{r_i}_k(j) \rightarrow 0 \ \mathrm{as} \ i \rightarrow \infty$.
	Thus, for $k \not \in \mathcal{K}$,
	there must exist a series of pixels $(j_k^r)_{r \in \mathbb{N}}$ such that $u_k^r(j_k^r) \ge \varepsilon > 0$ for $r$ sufficiently large. 
	Hence,
	\begin{align}
	(|f(j_k^r) - l^r(j_k^r) - c^r_k|)_{r \in \mathbb{N}} = (|f(j_k^r) - l_{\ker^\perp}^r(j_k^r) - \tilde c^r_k|)_{r \in \mathbb{N}}
	\end{align} 
	is bounded.
	Consequently, $(\tilde c_k^r)_{r \in \mathbb{N}}$ is bounded for all $k \not \in \mathcal{K}$.
	Together with $\tilde c_k^r = 0$ for $k\in \mathcal{K}$, we obtain that $(\tilde c^r)_{r \in \mathbb{N}}$ is bounded.
	Thus, the whole sequence $((\tilde u,\tilde c,\tilde l)^r)_{r \in \mathbb{N}}$ is bounded.
	
	It remains to show that it is an infimal one.		
	Note that adding $l^r_{\ker}$ to $c_k^r$ and subtracting it from $l^r(j)$ does not change the objective value.
	
	By construction, the constraint $\tilde u^r \in \triangle_K^n$ is still fulfilled.
	Thus, the only change in the objective value arises for the summands belonging to $k \in \mathcal{K}$ and $\hat k$.
	But since 
	\begin{align}
	\tilde u^r_{\hat k} - u^r_{\hat k} \rightarrow \mathbf{0} \ \mathrm{as} \ r \rightarrow \infty
	\end{align}
	and
	\begin{align}
	\sum_{j\in\mathcal{G}} u^r_k(j) \rightarrow 0 \ \mathrm{as} \ r \rightarrow \infty, \qquad k \in \mathcal{K},
	\end{align}
	the objective value does not change in the limit such that the sequence $((\tilde u,\tilde c,\tilde l)^r)_{r \in \mathbb{N}}$ is still an infimal sequence.
	This finishes the proof.
\end{proof}

\begin{remark} \label{rem:alternative}
Alternatively to \eqref{new_model_log} we could deal with the non logarithmic model
 \begin{align}\label{new_model}
	&\argmin_{u,C,L}  \Big\{ {\cal E}_{data}(u,C,L)
		+ \sum_{k=1}^K \lambda_k \Big\| | \nabla u_k |	\Big\|_1 
		+ \gamma \Big\| \nabla L 		\Big\|_2^2 \Big\}
		\\ 
	&\qquad \qquad \mathrm{subject \; to} \quad \left(u_k(j) \right)_{k=1}^K \in \triangle_K, 
		\;	j \in {\cal G}.
\end{align}
We prefer to work with the biconvex functional \eqref{new_model_log}, 
although we have obtained similar numerical results by minimizing the above triconvex functional.
\end{remark}

\section{Minimization Algorithm} \label{sec:palm}
%
In this section we deal with the minimization of our functional \eqref{new_model_func}.
First we mention that fixing $(c,l)$ we get
\begin{align}
E_{(c,l)} (u) \coloneqq \sum_{k=1}^K &\sum_{j \in {\cal G}} u_k(j)  |f(j) - l(j) - c_k|^2 \\
		&+ \sum_{k=1}^K \lambda_k \Big\| | \nabla u_k |	\Big\|_1
		+ \iota_{\triangle_K^n}(u),
\end{align}
which is convex, but not strictly convex.
Fixing $u$ we obtain the convex functional
$$
E_{u} (c,l) \coloneqq \sum_{k=1}^K \sum_{j \in {\cal G}} u_k(j)  |f(j) - l(j) - c_k|^2
+ \gamma  \Big\| \nabla l		\Big\|_2^2 .
$$
This functional becomes strictly convex if
we assume that we have only nonempty classes, i.e.,
$$
\sum_{j \in {\cal G}} u_k(j) \ge \varepsilon > 0
$$
and enforce the mean value of $l$ to be zero.
In other words,
\begin{align}\label{eq:star_strictlyconv}
\tilde E_{u} (c,l) \coloneqq &\sum_{k=1}^K \sum_{j \in {\cal G}} u_k(j)  |f(j) - l(j) - c_k|^2\\
&+ \gamma  \Big\| \nabla l		\Big\|_2^2 + \iota_{\{0\}}(\langle \mathbf{1}, l\rangle)
\end{align}
is strictly convex in $c$ and $l$ separately due to the quadratic term.
It is furthermore jointly convex in $(c,l)$ by the term $\iota_{\{0\}}(\langle \mathbf{1}, l\rangle)$.

For an algorithm which computes in an alternating way
\begin{align*}
u^{(r)} &\in \argmin_{u} E_{(c^{(r)} ,l^{(r)} )}(u) \\
(c^{(r+1)} ,l^{(r+1)} ) &\in  \argmin_{(c,l)} E_{u^{(r)}} (c,l )
\end{align*}
one can  prove similarly as in \cite{BHST87} that the resulting sequence 
$\{ u^{(r)},c^{(r)} ,l^{(r)} \}_r$ has a subsequence which
converges to a partial minimizer of \eqref{new_model_func}.
For an alternating convex search algorithm for general biconvex problems we refer to \cite{GPK07}.

In this paper we want to apply an algorithm which ensures the convergence
of the whole sequence $\{ u^{(r)},c^{(r)} ,l^{(r)} \}_r$.
To this end we need the definition of the proximal mapping.
For a proper, lower semi-continuous, convex function $h\colon \R^N  \rightarrow (-\infty,+\infty]$ and $\lambda >0$,
the proximal operator $\mathrm{prox}_{\lambda h}\colon \R^N  \rightarrow R^N$ is defined by
$$
\mathrm{prox}_{\lambda h}(x) \coloneqq \argmin_{y \in \R^N} \frac{1}{2\lambda} \|x-y\|_2^2 + h(y).
$$
Indeed, the minimizer of the right-hand side exists and is uniquely determined, see, e.g.,~\cite{Ro70}.

Recently, Bolte, Sabach and Teboulle \cite{BST14} considered 
problems of the form
\begin{align}\label{teboulle}
	\argmin_{x=(x_1,\ldots,x_m)} \Big\{ \sum_{i=1}^m h_i(x_i) + H(x) \Big\},
\end{align}
where $h_i\colon \R^{N_i} \rightarrow (-\infty,+\infty]$, $i=1,\ldots,m$,
are proper lower semi-continuous functions and
$H\colon \R^{N_1} \times \ldots \times \R^{N_m} \rightarrow \R$ is continuously differentiable.
Further, $\sum_{i=1}^m h_i(x_i) + H(x)$ needs to be a Kurdyka-\L ojasiewicz (KL) function.
For a discussion of KL functions, we refer to \cite{ABRS10,BDLM10,BST14}.
Here, we only want to emphasize that semi-algebraic functions are KL and our setting fits to the examples for semi-algebraic functions given in \cite[Sect.~5]{BST14}.
The authors suggested Algorithm \ref{PALM_general} for such problems.

\begin{algorithm*}
Input: functions $h_i\colon \R^{N_i} \rightarrow (-\infty,+\infty]$, $i=1,\ldots,m$,\\
\hspace{2.5cm} $H\colon \R^{N_1} \times \ldots \times \R^{N_m} \rightarrow \R$\\
\hspace{1.0cm} step-sizes $\tau_i^r$, $i=1,\ldots,m$, $r=0,1,\ldots$\\[1ex]
Initialization:  $x_i^{(0)} \in \R^{N_i}$, $i=1,\ldots,m$
\\[2ex]
For $r=0,1,...$ do
\begin{align}
	x^{(r+1)}_i &\in \mathrm{prox}_{\tau^r_i h_i} \left( x_i^{(r)} - \frac{1}{\tau_{i,r} } \nabla_{x_i} 
	H(x^{(r+1)}_1,\ldots,x^{(r+1)}_{i-1},x^{(r)}_i,\ldots,x^{(r)}_m) \right), \\
	& \hspace{8.6cm} i=1,\ldots,m.
\end{align}	
\caption{PALM for $m$ blocks}\label{PALM_general}
\end{algorithm*}

Under certain assumptions it was shown that the sequence $\{x^{(r)}\}_{r}$ converges to a critical point of the functional in \eqref{teboulle}.
%
\begin{theorem} \label{thm:teb}
Let $h_i\colon \R^{N_i} \rightarrow (-\infty,+\infty]$, $i=1,\ldots,m$
be proper lower semi-continuous functions
and 
$H\colon \R^{N_1} \times \ldots \times \R^{N_m} \rightarrow \R$  a continuously differentiable function such that $\nabla H$ is Lipschitz continuous on bounded sets.
Let $h_i$, $i=1,\ldots,m$ and $\sum_{i=1}^m h_i+H$ be bounded from below.
Further, let $\sum_{i=1}^m h_i(x_i) + H(x)$ be a KL function.
Assume that the sequence of iterates $\{x^{(r)}\}_{r}$ produced by PALM is bounded.
Further suppose that for $i=1,\ldots,m$,
	the partial gradients $\nabla_{x_i} H(x)$ are globally Lipschitz for 
	fixed $$(x_1,\ldots,x_{i-1},x_{i+1},\ldots,x_m)$$ with Lipschitz constant 
	$$L_i(x_1,\ldots,x_{i-1},x_{i+1},\ldots,x_m),$$
	where  
	\begin{equation}
	L_i \left( (x_1^{(r)},\ldots,x_{i-1}^{(r)},x_{i+1}^{(r)},\ldots,x_m^{(r)}) \right) \le \lambda_i^+
	\end{equation}
	for all $r \in \mathbb N$ and some 	$\lambda_i^+ \in \R$. 
Then, for 
$$
\tau_{i,r} \ge L_i \left( (x_1^{(r)},\ldots,x_{i-1}^{(r)},x_{i+1}^{(r)},\ldots,x_m^{(r)}) \right), \quad i=1, \ldots,m,
$$ 
the sequence $\{x^{(r)}\}_{r}$ converges to a critical point of the functional in \eqref{teboulle}.
\end{theorem}
%

For minimizing \eqref{new_model_func} we apply PALM with $m=3$ blocks and the following setting
\begin{align}\label{eq:PALMsetting_add}
h_1(u) &\coloneqq \sum_{k=1}^K \lambda_k \big\| |\nabla u_k|\big\|_1 + \iota_{\triangle_K^n}(u),\\
h_2(c) &\coloneqq 0,\\
h_3(l) &\coloneqq \iota_{\{0\}}(\langle \mathbf{1}, l\rangle),\\
H(u,c,l) &\coloneqq \sum_{j,k} u_k(j)(f(j) - l(j) - c_k)^2 + \gamma \|\nabla l\|_2^2.
\end{align}
More precisely,
\begin{align}
	H(u,c,l) + h_1(u) + &h_2(c) + h_3(l) \\
	&= E(u,c,l) + \iota_{\{0\}}(\langle \mathbf{1}_N, l\rangle),
\end{align}
i.e., we have modified the functional \eqref{new_model_func} by the additional function $h_3$, which
enforces the logarithmic illumination to have mean value zero, see \eqref{eq:star_strictlyconv}.
Note that adding a constant $C\in \R$ to  $l(j)$ for all $j \in {\cal G}$,
and subtracting this constant from the $c_k$, $k=1,\ldots,K$
does not change the value of the functional \eqref{new_model_func}.
Therefore we have an ambiguity in the solution consisting of a constant function.
By introducing $h_3$ we enforce that $l$ has mean zero, i.e.,
we fix the constant to $C \coloneqq -\langle \mathbf{1}, l \rangle /n$ such that the ambiguity is removed.

Of course it would also be possible to use PALM with 
$m=2$ blocks by handling $(c,l)$ together.
However, this leads to joint estimates for the Lipschitz constant in the partial gradient of $H$
with respect to $(c,l)$ which is more restrictive than the Lipschitz constants from the
separate gradients with respect to $c$ and $l$, see also proof of Corollary \ref{cor:palm}.

In PALM we need the following partial gradients of $H$:
\begin{align}\label{eq:PALM_grads}
	\nabla_u H(u,c,l) &= \left( (f(j) - l(j) - c_k)^2 \right)_{j,k},
	\\
	\nabla_c H(u,c,l) &= 2 \left( \sum_j u_k(j) (c_k + l(j) - f(j)) \right)_{k},
	\\
	\nabla_l H(u,c,l) &= 2 \left(\sum_k u_k(j) (l(j) + c_k - f(j))\right)_j \\
	& \qquad + 2 \gamma \nabla^* \nabla l.
\end{align}
Based on these gradients PALM reads for our setting as follows:

\begin{algorithm*}
	Input: step-sizes $\tau_i^r \in \R_{>0}$, $i=1,2,3$, $r=0,1,\ldots$\\[1ex]
Initialization:  $u^{(0)} \in \R^{nK}$, $c^{(0)} \in \R^K$, $l^{(0)} \in \R^n$
\\[2ex]
For $r=0,1,...$ do
\begin{align}
	u^{(r+1)} &= \mathrm{prox}_{\tau_{1,r} h_1} \left( u^{(r)} - \frac{1}{\tau_{1,r}} \nabla_{u} 	H(u^{(r)},c^{(r)},l^{(r)}) \right),
	\\	
	c^{(r+1)} &= c^{(r)} - \frac{1}{\tau_{2,r}} \nabla_{c} H(u^{(r+1)},c^{(r)},l^{(r)}),
	\\	
	l^{(r+1)} &= \argmin_l \iota_{\{0\}}(\langle \mathbf{1}, l \rangle) + \frac12 \| l^{(r)} 
	- \frac{1}{\tau_{3,r}} \nabla_{l} 	H(u^{(r+1)},c^{(r+1)},l^{(r)}) - l\|_2^2.
\end{align}	
\caption{PALM for \eqref{eq:PALMsetting_add}.}\label{alg:PALM_special}
\end{algorithm*}

The first proximum $\mathrm{prox}_{\tau_{1,r} h_1}$ is not given analytically,
but can be computed by several methods. 
Here we use the primal dual algorithm proposed in \cite{CP11}.
The second step is just a gradient descent step.
The last proximum is a projection 
of $a \coloneqq l^{(r)} - \frac{1}{\tau_{3,r}} \nabla_{l} 	H(u^{(r+1)},c^{(r+1)},l^{(r)})$
onto the hyperplane
 $\{l: \langle \mathbf{1} , l \rangle = 0\}$
and can be computed
by subtracting from $a$ its mean value.

\begin{corollary} \label{cor:palm}
 	For the functions $h_i$, $i=1,2,3$ and $H$ defined in \eqref{eq:PALMsetting_add} and
	\begin{align}
		\tau_{1,r} > 0,\quad
		\tau_{2,r} \ge 2 n,\quad
		\tau_{3,r} \ge 2 + 8 d\gamma,
		\qquad r=0,1,\ldots,
	\end{align}
	Algorithm \ref{alg:PALM_special} converges to a critical point of \eqref{new_model_func}.
\end{corollary}
\begin{proof}
We have to check that the assumptions of Theorem \ref{thm:teb} are fulfilled.

It is easy to check that $H(u,c,l)+h_1(u)+h_2(c)+h_3(l)$ is a semi-algebraic function and therefore a KL function, see the examples in \cite[Sect.~5]{BST14}.
The functions $h_1, h_2, h_3$ are lower semi-continuous and bounded from below.
Since $H$ is twice continuously differentiable, its gradient is Lipschitz on bounded sets
and by $u \in \triangle_K^n$  the function is also bounded from below.		

Let us consider the partial gradients of $H$ given by \eqref{eq:PALM_grads}.
Since $\nabla_u H(u,c,l)$ does not depend on $u$, we get immediately 
$L_1(c, l) = 0$.
For  the gradient with respect to $c$ it follows with $u \in \triangle_K^n$ that
	\begin{align}
		\|\nabla_c H(u,c_1,l) - \nabla_c H(&u,c_2,l) \|_2  \\
		&= \| \big(2\sum_j u_k(j) (c_1 - c_2)\big)_{k} \|_2
		\\
		& \le 2 n \|c_1 - c_2\|_2.
	\end{align}
Finally,  $\nabla_l H(u,c,l)$ is Lipschitz since
		\begin{align}
		\| \nabla_l H(u,c,l_1) - \nabla_l H(&u,c,l_2)\|_2 \\
		 =& \| \big(2\sum_k u_k(j) 
			(l_1(j) -l_2(j))\big)_j \\
		 &+ 2 \gamma \nabla^*\nabla (l_1 - l_2)\|_2\\
			\le& C(u) \|l_1 - l_2\|_2,
	\end{align}
	where 
	$$C(u) \coloneqq 2 \max_{k,j}{\{|u_k(j)|\}} + 2 \gamma \|\nabla^* \nabla\|.$$
	The $d$-dimensional Laplacian $\nabla^* \nabla$ has a spectral norm smaller than $4d$. This can be seen as follows.
	In 1D one simply has the forward difference matrix with mirror boundary conditions 
	\begin{align}
		\nabla = D_{n_1} \coloneqq \begin{pmatrix} -1 & 1 & \\ & \ddots & \ddots \\ & & -1 & 1 \\ & & & 0 \end{pmatrix} \in \R^{n_1,n_1}.
	\end{align}
	The eigenvalues of $\nabla^*\nabla$ are $4 \sin^2 ( \frac{\pi j}{2n})$, $j=0,\ldots,n_1-1$, see, e.g., \cite{SM14}. 
	Hence, $\|\nabla^* \nabla\| < 4$.
	In $d$ dimensions, after reshaping the $d$-dimensional image into a vector, 
	the gradient becomes
	$$
	\nabla = \begin{pmatrix} D_{1} \\ \vdots  \\  D_{d} \end{pmatrix}, \quad D_{i} \coloneqq I_{b_i} \otimes D_{n_i} \otimes I_{a_i} 
	$$
	where $a_i \coloneqq \Pi_{j=1}^{i-1} n_j$, $b_i \coloneqq \Pi_{j=i+1}^d n_j$ and $\otimes$ denotes the tensor product of matrices.
	Then
	$$
	\nabla^*\nabla = \sum_{i=1}^d I_{b_i} \otimes D_{n_i}^\tT D_{n_i} \otimes  I_{a_i} .
	$$
	Since the eigenvalues of $A \otimes B$ are given by the products of the eigenvalues of $A$ and $B$, one obtains 
	$\| \nabla^*\nabla\| \le \sum_{i=1}^d \| D_{n_i}^\tT D_{n_i}  \| \le 4d$.
	Since $u^{(r)}\in \triangle_K^n$,
	we see that $L_3(u^{(r)}, c^{(r)}) \le 2+ 8d \gamma$.
	
	Finally, the sequence of iterates produced by the algorithm is bounded by the following arguments.
	By adding the constraint to $l$, i.e., $\mathbf{1}^\tT l = 0$, the resulting functional becomes coercive.
	Together with the decrease property for the PALM iterates \cite[Lemma 3.3]{BST14} one immediately gets the boundedness.
\end{proof}
%

\section{Numerical Examples} \label{sec:numerics}
In this section we demonstrate the performance of our algorithm.
The algorithm was implemented in MATLAB with the following general parameter setting:
\begin{itemize}
 \item Lipschitz constants in PALM:
According to the Lipschitz constants we can choose any $\tau_1^r >0$.
In practice $\tau_1^r$ needs to be chosen rather small to achieve an acceptable convergence speed, here we set $\tau_{1,r} \coloneqq 10^{-6}$.
Further we set $\tau_{2,r}\coloneqq n$.
Although this choice does not fulfill the restrictions given by the Lipschitz constants, we observed numerical convergence. 
In particular compared to a smaller stepsize that fulfills the Lipschitz condition, we got the same results in a faster way.
Finally, we set $\tau_{3,r}\coloneqq2+8d\gamma$ according to the Lipschitz condition.
\item Initialization:	
The illumination $l$ is initialized as a very smooth version of the input image obtained by
convolving the input image with a Gaussian kernel of large standard deviation $\sigma$ displayed in the caption of the images.
The codebook $c$ is initialized with a rough initial guess, and $u$ with constant entries $\frac{1}{K}$.
We observed that the initialization of $u$ and $c$ is not important. For reasonable initial values the results were almost equal.
\item Iteration number: 
In all our experiments we performed $2000$ (outer) iterations of PALM. 
In each outer iteration, we applied $50$ (inner) iterations 
of the primal dual method for computing the proximal operator in the first step of PALM.
If we would apply the inner PDHG iterations on their own, usually more than $50$ iterations would be necessary.
Here, we observed that $50$ inner iterations are sufficient since the initialization is already close to the solution after some outer iterations.
\item Regularization parameters: 
If not stated otherwise, the other parameters, i.e., $\lambda_k$, $k=1,\ldots,K$ and $\gamma$ 
were chosen according to the best visual impression and are stated in the captions of the corresponding figures.
\end{itemize}
\noindent
We compare our approach with the following three other methods:
\begin{itemize}
	\item[M1)]
		Model \eqref{model_li} proposed in \cite{LHDGMG11} for 2D images,
		where we used the program package of the authors available at \url{http://www.engr.uconn.edu/~cmli/code/}.
	\item[M2)]
		We apply the pure segmentation model
		\begin{align*}
			&\argmin \Big\{ \sum_{k=1}^K \sum_{j \in \mathcal{G}} u_k(j) |f(j)  - c_k|^2
			+  \lambda \sum_{k=1}^K \| \nabla u_k \|_{2,1} 
			\Big\}
			\\ 
			&\qquad	\qquad \mathrm{subject \; to}  \quad u(j) \in \triangle_K, 
			\;	j \in \mathcal{G},
		\end{align*}
		proposed, e.g., in \cite{SS12a}. This model does not take care of illumination changes. We use this method for comparison to show that it is necessary to include the illumination in the model.
	\item[M3)]
	We apply a two step procedure.
		In the first step we use the initialization described above to estimate the illumination
		and correct the image by dividing it by the illumination.
		In the second step we apply model M2) to the corrected image.
\end{itemize}
We start with an artificial 2D image as a ground truth experiment.
Next we apply our algorithm to 2D medical images, in particular CT images, 
since this kind of images was used to test the algorithm M1) in \cite{LHDGMG11}.
Finally, we consider 3D FIB tomography data which was the motivation for this work.
Since some of the images have pixels with value zero, 
we add a small constant to avoid problems when taking the logarithm of such images.

\paragraph{Artificial 2D Images}
We start with the artificial 2D image with varying illumination in Fig.~\ref{fig:artificial_a}.
The figure shows the noise-free case.
A three class segmentation by M2) 
leads to completely wrong results.
Our model is able to segment this image exactly and extract the proper illumination $l$.
Similar results can be obtained by the method M1) \cite{LHDGMG11}.

\begin{figure}
	\centering
	\includegraphics[width=0.24\textwidth]{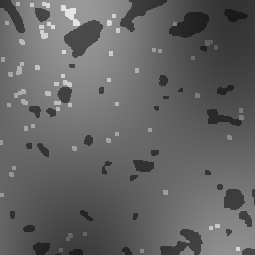}
	\includegraphics[width=0.24\textwidth]{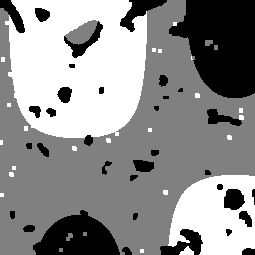}
	\includegraphics[width=0.24\textwidth]{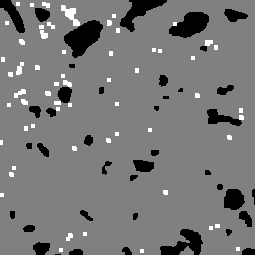}
	\includegraphics[width=0.24\textwidth]{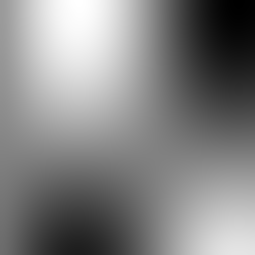}
	\caption{From left to right: original image, segmented image by M2) (without illumination correction), 
		segmented image with our method ($\lambda_1=\lambda_2=0.2$, $\gamma=100$, $\sigma=30$) and the computed illumination by our model. 
	}
	\label{fig:artificial_a}
\end{figure}

Next, we add Gaussian noise of standard deviation $s$ 
and compare the segmentation of our model and M1).
For our model we optimized the TV parameter $\lambda = \lambda_1 = \lambda_2$ for each noise-level by a grid search according to the smallest number of wrongly assigned pixels,
and took $\gamma=100$  as in the noise-free case.
In the model \eqref{model_li} from \cite{LHDGMG11} we set $\sigma=7$ which also provided the best result in the noise-free case.
As in our model, the parameter $\lambda$ was optimized by a grid search for each noise-level. Additionally, we optimized the parameter $\gamma$ by a grid search.
Fig.~\ref{fig:artificial_plot}  shows the percentage of wrongly assigned pixels depending on $s$ for our model and M1).
Up to the noise level $s = 0.004$ the proposed method clearly outperforms M1).
For $s = 0.005$ however, one class vanishes due to the noise so that the number of wrongly assigned pixels increases significantly.
To avoid this the third center $c_3$, which belongs to the vanishing class, can be fixed.
Then the proposed method outperforms M1) also for the higher noise levels.
Fig.~\ref{fig:artificial_noise} depicts the results for the noise level $s = 0.005$.
\begin{figure}
	\centering
	\includegraphics[width=0.24\textwidth]{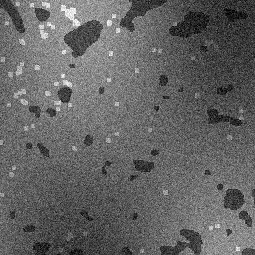}
	\includegraphics[width=0.24\textwidth]{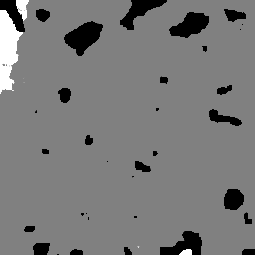}
	\includegraphics[width=0.24\textwidth]{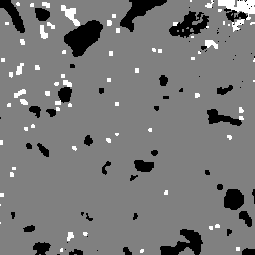}
	\includegraphics[width=0.24\textwidth]{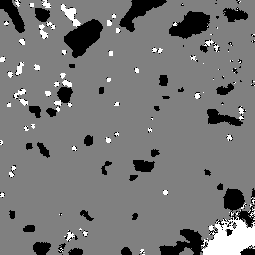}
	\caption{From left to right: noisy image ($s=0.005$), segmented images by our method ($\lambda=0.8$), our method with fixed $c_3$ ($\lambda=0.5$) and M1).}
	\label{fig:artificial_noise}
\end{figure}

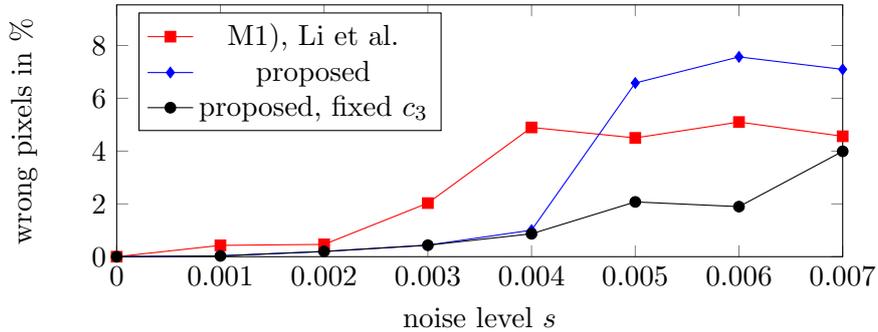
\begin{figure}
	\centering
			\begin{tikzpicture}
			\begin{axis}[
			width=.88\linewidth, height=140pt,
			xmin=0, xmax=7,
			xtick={0,1,2,3,4,5,6,7},
			xticklabels={$0$,$0.001$,$0.002$,$0.003$,$0.004$,$0.005$,$0.006$,$0.007$},
			xlabel={noise level $s$},
			ymin=0, ymax=6201,
			ytick={0,1300,2600,3901,5201},
			yticklabels={$0$,$2$,$4$,$6$,$8$},
			ylabel={wrong pixels in $\%$},
			legend pos=north west
			]
			
			\addplot[mark=square*,color=red] plot coordinates{ 
				(0,0)    		(1,279)    (2,303)    (3,1320)  
				(4,3183)    (5,2925)    (6,3316)    (7,2966)
				};					
			\addlegendentry{M1), Li et al.};	               
			
			\addplot[mark=diamond*,color=blue] plot coordinates{ 
				(0,0)    		(1,21)    (2,127)    (3,285)  
				(4,653)    (5,4278)    (6,4920)    (7,4613)
			};						
			\addlegendentry{proposed};	               
			
			\addplot[mark=*] plot coordinates{
				(0,0)    		(1,20)    (2,130)    (3,283)  
				(4,564)    (5,1349)    (6,1233)    (7,2595)
			};								                  				
			\addlegendentry{proposed, fixed $c_3$};					
			\end{axis}
			\end{tikzpicture}
	\caption{Percentage of wrongly assigned pixels depending on the noise level using the model M1), the proposed method and the proposed method with fixed third entry of the codebook.}
	\label{fig:artificial_plot}
\end{figure}

\paragraph{Medical Images}
Next we present results for different medical images.
The images shown in Fig.~\ref{fig:medical} were taken from \cite{LHDGMG08,LKGD07}.

The first row shows results for an X-ray image of bones.
The result of our model is compared to the result of M1)
 with the parameters proposed in the paper \cite{LHDGMG11}, i.e., $\sigma=4, \gamma=1$ and $\lambda= 0.001\cdot 255^2$. 
The white part corresponds to positive values of the resulting function of the level set method M1).
Furthermore, we depict the result of the two-step method M3) 
to show that correcting the illumination only in the preprocessing step 
before the segmentation is not sufficient. 
The methods M1) and M3) perform equally well, where each one shows different slight artifacts.
The proposed model gives the best segmentation result, in particular the left bone and the upper part of the right bone are segmented correctly.

The second and third row show results for two CT angiography images of vessels.
It is clearly necessary to incorporate the illumination since parts of the images differ widely in their brightness. 
For the first vessel image the parameters proposed in \cite{LHDGMG11} were used.
For the second vessel we set $\sigma=10, \gamma=10$.  
The first segmented vessel by our method shows slightly thicker structures.
In the image center the segmentation method M1) therefore gives better results.
In the lower part M1) leads to too thin structure. Here our method performs better.
M3) leads to the worst results in this example. 
Many thin structures are not corrected and some additional artifacts are visible.
For the second vessel all three methods give good results.
However, M1) and M3) show some small artifacts that are not visible with the proposed method.

\begin{figure}
	\centering
	\includegraphics[width=0.19\linewidth]{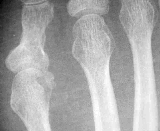}
	\includegraphics[width=0.19\linewidth]{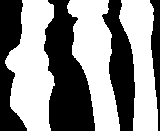}
	\includegraphics[width=0.19\linewidth]{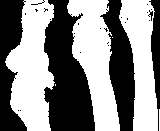}
	\includegraphics[width=0.19\linewidth]{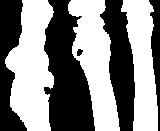}\\[1.5ex]
	\includegraphics[width=0.19\linewidth]{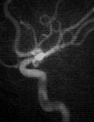}
	\includegraphics[width=0.19\linewidth]{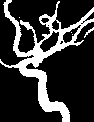}
	\includegraphics[width=0.19\linewidth]{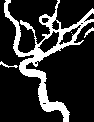}
	\includegraphics[width=0.19\linewidth]{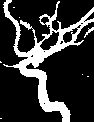}\\[1.5ex]
	\includegraphics[width=0.19\linewidth]{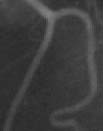}
	\includegraphics[width=0.19\linewidth]{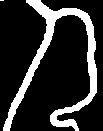}
	\includegraphics[width=0.19\linewidth]{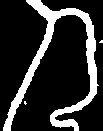}
	\includegraphics[width=0.19\linewidth]{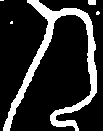}
	\caption{From left to right: original image, segmented image by our method, results by M1) and M3).
		First row: X-ray image of bones  ($\lambda_1=\lambda_2=0.5$, $\gamma=25$, $\sigma=20$), 
		second row: CTA image of a vessel  ($\lambda_1=\lambda_2=0.025$, $\gamma=25$, $\sigma=20$), 
		third row: CTA image of a vessel  ($\lambda_1=\lambda_2=0.5$, $\gamma=50$, $\sigma=10$).}
	\label{fig:medical}
\end{figure}

\paragraph{3D FIB Tomography Images} \label{subsec:segm:FIB}
Next we are interested in the segmentation of 3D images stemming from FIB tomography.
We consider 3D data of an aluminum matrix composite.
More precisely, the material consists of three phases which we want to segment: 
the first phase is aluminum, the second one consists of silicon carbide (SiC) particles and the third one are copper aggradations.
The aluminum particles are darker than the surrounding aluminum matrix.
The copper aggradations are visible as small bright spots.

In Fig.~\ref{fig:real_3d} and \ref{fig:real_3d_b} two 3D data sets of this material with different particle sizes are examined.
The necessity to take the illumination into account when segmenting the image in Fig.~\ref{fig:real_3d} was already discussed in Fig.~\ref{fig:bad_ex}.
In Fig.~\ref{fig:real_3d:a}-\subref{fig:real_3d:b} and Fig.~\ref{fig:real_3d_b:a}-\subref{fig:real_3d_b:b}, respectively, an exemplary slice of the data set is shown together with the slice of the segmented data. 
In addition a visualization of the segmented 3D data set is provided in Fig.~\ref{fig:real_3d:c}-\subref{fig:real_3d:d} and Fig.~\ref{fig:real_3d_b:c}-\subref{fig:real_3d_b:d}, respectively.
Note that due to the very different size of SiC and copper we had to choose a small parameter for the aggradations ($\lambda_3$) and a larger one for the particles ($\lambda_2$) to segment both correctly.
Furthermore we fixed $c_3$, which corresponds to the copper aggradations, to avoid that this small segment vanishes similar as in Fig.~\ref{fig:artificial_noise}.

\begin{figure}
\centering
\begin{subfigure}[t]{0.48\textwidth}\centering
\includegraphics[width=0.8\linewidth]{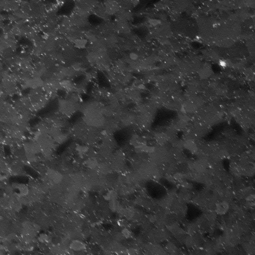}
\caption{Exemplary slice}
\label{fig:real_3d:a}
\end{subfigure}
\begin{subfigure}[t]{0.48\textwidth}\centering
\includegraphics[width=0.8\linewidth]{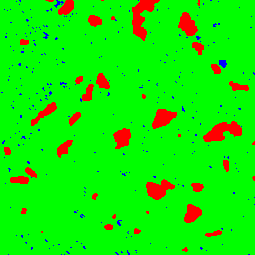}
\caption{Segmented slice}
\label{fig:real_3d:b}
\end{subfigure}
\begin{subfigure}[t]{0.48\textwidth}\centering
\includegraphics[width=0.98\linewidth]{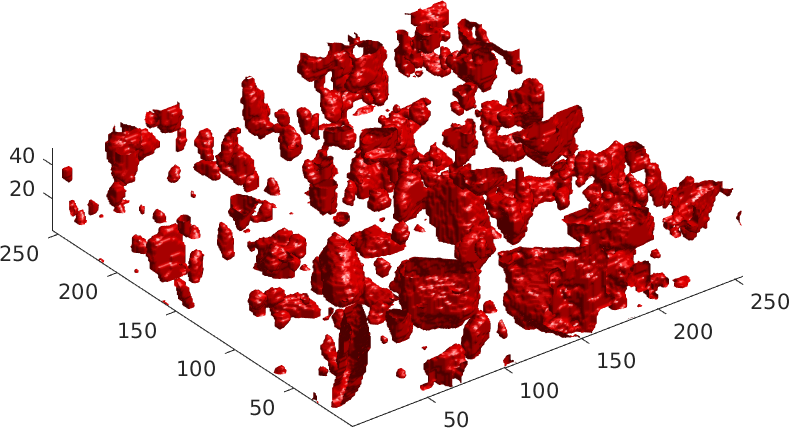}
\caption{3D visualization of the SiC segment}
\label{fig:real_3d:c}
\end{subfigure}
\begin{subfigure}[t]{0.48\textwidth}\centering
\includegraphics[width=0.98\linewidth]{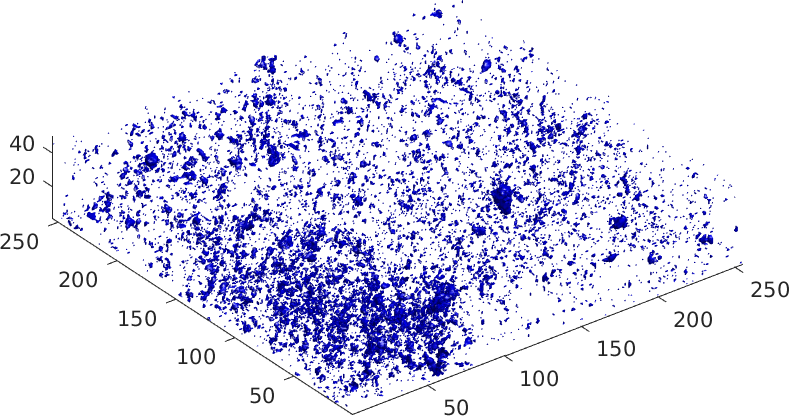}
\caption{3D visualization of the copper segment}
\label{fig:real_3d:d}
\end{subfigure}
\caption{
Segmentation of 3D FIB tomography data with three classes by our method ($\lambda_1=0.1, \lambda_2=0.8, \lambda_3=0.1$), red: SiC particles, blue: copper aggradations. 
}
\label{fig:real_3d}
\end{figure}

\begin{figure}
	\centering
	\begin{subfigure}[t]{0.48\textwidth}\centering
		\includegraphics[width=0.8\linewidth]{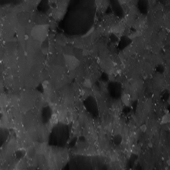}
		\caption{Exemplary slice}
		\label{fig:real_3d_b:a}
	\end{subfigure}
	\begin{subfigure}[t]{0.48\textwidth}\centering
		\includegraphics[width=0.8\linewidth]{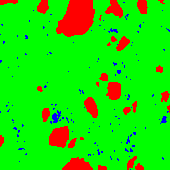}
		\caption{Segmented slice}
		\label{fig:real_3d_b:b}
	\end{subfigure}
	\begin{subfigure}[t]{0.48\textwidth}\centering
		\includegraphics[width=0.98\linewidth]{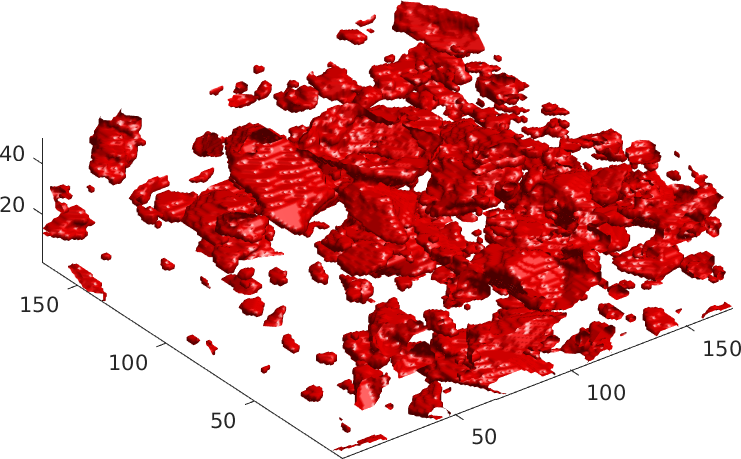}
		\caption{3D visualization of the SiC segment}
		\label{fig:real_3d_b:c}
	\end{subfigure}
	\begin{subfigure}[t]{0.48\textwidth}\centering
		\includegraphics[width=0.98\linewidth]{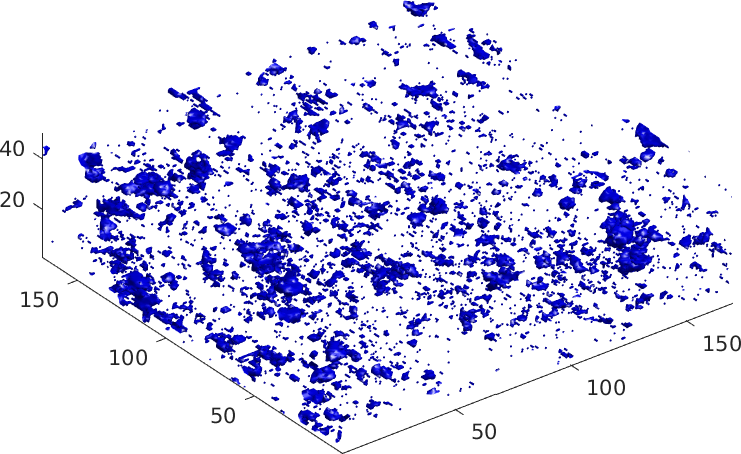}
		\caption{3D visualization of the copper segment}
		\label{fig:real_3d_b:d}
	\end{subfigure}
\caption{Segmentation of 3D FIB tomography data with three classes by our method ($\lambda_1=0.1, \lambda_2=0.6, \lambda_3=0.1$), red: SiC particles, blue: copper aggradations. 
}
\label{fig:real_3d_b}
\end{figure}

\section{Conclusions} \label{sec:conclusions}
We proposed a novel biconvex  model 
for segmentation of images with intensity inhomogeneities. 
that combines a total variation based approach for segmentation with a multiplicative model for the intensity inhomogeneities.
This results in a simultaneous segmentation and intensity correction algorithm without preprocessing.
For the minimization of the resulting functional we used the PALM algorithm for which we could prove convergence to a critical point.
We applied our model very successfully to 3D FIB tomography images and
showed that the proposed model leads to good results compared to a state-of-the-art method for various 2D medical images. 
In future work, we want to address the question how good the relaxed two-class model solves the original binary problem with
$u_k(j) \in \{0,1\}$, see \cite{NEC06}.
The incorporation of the illumination into a segmentation framework  
with additional linear operators as, e.g., blur, see \cite{CCZ13} appears also to be useful.
Finally, we want to handle  other than gray-valued images.
\\[1ex]

{\bf Acknowledgement.}
Funding by the German Research Foundation (DFG) 
with\-in the Research Training Group 1932 "Stochastic Models for Innovations in the Engineering Sciences", project area P3, 
is gratefully acknowledged.

The authors thank K.~Schladitz (Fraunhofer ITWM, Kai\-serslautern) and 
F.~Balle, T.~Beck and S.~Schuff (Department of Mechanical and Process Engineering, 
University of Kaiserslautern)
for fruitful discussions.
We are thankful to T.~Löber (Nano Structuring Center Kaiserslautern) for creating FIB/SEM data of the aluminum matrix composite used in Figure \ref{fig:bad_ex}, \ref{fig:real_3d} and \ref{fig:real_3d_b}. 

\bibliography{refs_infconv}
\bibliographystyle{abbrv}

\end{document}